\documentclass[11pt,reqno]{amsart}

\usepackage{amsmath}
\usepackage{amsfonts}
\usepackage{mathtools}
\usepackage{amssymb}
\usepackage{comment}
\usepackage{enumerate}
\usepackage{amsthm}
\usepackage{tikz-cd}
\usepackage[a4paper,
            bindingoffset=0.2in,
            left=1in,
            right=1in,
            top=1in,
            bottom=1in,
            footskip=.25in]{geometry}
\allowdisplaybreaks

\newtheorem{thm}{Theorem}
\newtheorem{lemma}[thm]{Lemma}

\newtheorem{ex}[thm]{Example}
\newtheorem{rmk}[thm]{Remark}
\theoremstyle{definition}

\DeclareMathOperator{\ord}{ord}
\DeclareMathOperator{\Log}{Log}
\DeclareMathOperator{\Tr}{Tr}

\newcommand{\Z}{\mathbb{Z}}
\newcommand{\Q}{\mathbb{Q}}

\newcommand{\C}{\mathbb{C}}
\newcommand{\N}{\mathbb{N}}

\newcommand{\M}{\mathcal{M}}
\newcommand{\T}{\mathcal{T}}
\newcommand{\cQ}{\mathcal{Q}}

\renewcommand{\i}{\infty}
\renewcommand{\t}{\tau}
\renewcommand{\H}{\mathbb{H}}

\newcommand{\G}{\Gamma}
\newcommand{\g}{\gamma}
\renewcommand{\c}{\mathcal{C}}

\newcommand{\SL}{\mathrm{SL}_{2}}

\newcommand{\RNum}[1]{\uppercase\expandafter{\romannumeral #1\relax}}
\newcommand{\rnum}[1]{\lowercase\expandafter{\romannumeral #1\relax}}
\renewcommand{\div}{{\rm div}}

\numberwithin{equation}{section}
\numberwithin{thm}{section}

\begin{document}

\title{Multiplicative Hecke operators and their applications \MakeUppercase{\romannumeral2}}

\author[Chang Heon Kim and Gyucheol Shin]{Chang Heon Kim and Gyucheol Shin$^{*}$}

\address{Department of Mathematics, Sungkyunkwan University, Suwon 16419, Korea}
\email{chhkim@skku.edu}
\address{Department of Mathematics, Sungkyunkwan University, Suwon 16419, Korea}
\email{sgc7982@gmail.com}

\begin{abstract}
Inspired by Borcherds' questions \cite[17.10]{Bor}, Guerzhoy constructed a new type of Hecke operators $\T(p)$, called the multiplicative Hecke operators, which acts on the space of meromorphic modular forms on the full modular group $\SL(\Z)$. In \cite{KS}, this result was extended in two directions: to higher levels and to $\T(n)$ with a positive integer $n$. In this paper, building on the results in \cite{KS}, we further generalize the result in another direction by considering alternative infinite product expansions of meromorphic modular forms. As an application, we demonstrate how multiplicative Hecke operators relate both the divisor of modular forms and traces of singular moduli. Additionally, we prove the existence of a modular form with nonintegral coefficients whose poles or zeros are only supported at the cusps and which is not a multiplicative Hecke eigenform.
\end{abstract}

\maketitle

\renewcommand{\thefootnote}%
             {}
 {\footnotetext{

This work was supported by the National Research Foundation of Korea(NRF) grant 
funded by the Korea government(MSIT)(RS-2024-00348504).}

\section{Introduction and statement of results}
Let
\begin{equation*}
\G_{0}(N):=\{\begin{psmallmatrix}a&b\\c&d\end{psmallmatrix}\in\SL(\Z):c\equiv0\pmod{N}\}.
\end{equation*}
A meromorphic function on $\H^{*}:=\{\t\in\C:{\rm Im}(\t)>0\}\cup \mathbb{P}^{1}(\Q)$ is called a \textit{meromorphic modular form of weight $k$ on $\G_{0}(N)$} if it satisfies $(f|_{k}\g)(\t):=(c\t+d)^{-k}f(\g\t)=f(\t)$ for all $\g=\begin{psmallmatrix}a&b\\c&d\end{psmallmatrix}\in\G_{0}(N)$ where $\g\t:=\frac{a\t+b}{c\t+d}$.
Hecke operators are linear operators acting on the space of modular forms on $\G_{0}(N)$. For each prime $p$ and modular form $f$ of weight $k$ on $\G_{0}(N)$, the (usual) Hecke operator $T(p)$ is defined by
\begin{equation*}
f|T(p):=
\begin{cases}
p^{k/2-1}\sum\limits_{\substack{ad=p\\0\leq b<d}}f|_{k}\begin{psmallmatrix}a&b\\0&d\end{psmallmatrix}, & \text{ if } p\nmid N,
\\
p^{k/2-1}\sum\limits_{j=0}^{p-1}f|_{k}\begin{psmallmatrix}1&j\\0&p\end{psmallmatrix}, & \text{ if } p|N.
\end{cases}
\end{equation*}

Inspired by Borcherds' questions \cite[17.10]{Bor}, Guerzhoy \cite{Guer} introduced a new class of Hecke operators $\T(p)$ (for prime $p$) on $\SL(\Z)$, called the \textit{multiplicative Hecke operator}, acting on the multiplicative group of meromorphic modular forms on $\SL(\Z)$. He showed that the Borcherds isomorphism commutes with the multiplicative Hecke operators and the usual Hecke operators acting on the space of modular forms of weight 1/2 on $\G_{0}(4)$. More recently, in \cite{JKK}, Jeon, Kang and the first author generalized Guerzhoy's result \cite[Theorem 1]{Guer} to higher level cases. They proved that the generalized Borcherds products defined in \cite[Theorem 6.1]{BO} commute with the multiplicative Hecke operators $\T(p)$ and the usual Hecke operators on half integral weight vector-valued harmonic weak Maass forms. In addition, they showed that the logarithmic derivative also commutes with the multiplicative Hecke operators and the usual Hecke operators acting on the space of meromorphic modular forms of weight 2. These results, however, were established only for primes $p$ not dividing the level $N$. In our previous work \cite{KS}, we extended it to $\T(n)$ for all positive integers $n$ and to higher levels $N$ simultaneously. To be more precise, let $N$ be a positive integer, and let $\M(N)$ denote the multiplicative group of meromorphic modular forms on $\G_{0}(N)$ with a unitary multiplier system, integer Fourier coefficients, and leading coefficient 1. We let $\M_{k,h}(N)$ denote the subset of $\M(N)$ consisting of meromorphic modular forms of weight $k$ satisfying $v_{\i}(f)=h$. Here $v_{\i}(f)$ denotes the vanishing order of $f$ at the cusp $\i$. Following \cite{KS} for each prime $p$ and $f\in\M_{k,h}(N)$, the multiplicative Hecke operator $\T(p)$ acting on $\M(N)$ is defined by 
\begin{equation*}
f|\T(p):=
\begin{cases}
\mu\prod\limits_{\substack{ad=p\\0\leq b<d}}f|_{k}\begin{psmallmatrix}a&b\\0&d\end{psmallmatrix}, & \text{ if } p\nmid N,
\\
\mu\prod\limits_{j=0}^{p-1}f|_{k}\begin{psmallmatrix}1&j\\0&p\end{psmallmatrix}, & \text{ if } p|N,
\end{cases}
\end{equation*}
where $\mu$ is a constant, depending on $k$, $p$, and a multiplier system of $f$, chosen such that the leading coefficient of $f|\T(p)$ is equal to one. In \cite{KS}, we showed that this new type of Hecke operator can be naturally extended to $\T(n)$, where $n\in\N$, similar to the case of the usual Hecke operators (for more details, see \cite[Theorem 1.2 and Definition 1.6]{KS}). Furthermore, we computed the explicit expression of the exponents in the infinite product expansion of $f|\T(n)$. Moreover, we showed that the following diagram is commutative, which completely covers the gap between Guerzhoy's answer \cite[Theorem 1]{Guer} and Borcherds' question \cite[17.10]{Bor}.

\begin{equation}\label{comm}
\begin{tikzcd}
H_{1/2,\tilde{\rho}_{N}}'\arrow{r}{B}\arrow{d}{nT_{1/2}(n^{2})}&\M^{H}(N)\arrow{r}{\mathcal{D}}\arrow{d}{\T(n)}&M_{2}^{mer}(N)\arrow{d}{T_{2}(n)}
\\
H_{1/2,\tilde{\rho}_{N}}'\arrow{r}{B}&\M^{H}(N)\arrow{r}{\mathcal{D}}&M_{2}^{mer}(N)
\end{tikzcd}
\end{equation}

For the numerous of notations appearing in the above diagram, we refer to \cite{JKK} or \cite{KS}.
\\

\subsection{Multiplicative Hecke operators and alternative infinite product expansions of modular forms}
\indent The first purpose of this paper is, as a subsequent work to \cite{KS}, to provide another description of multiplicative Hecke operators by utilizing an alternative expression for the infinite product expansion of modular forms. Before stating our results, following \cite{Zagier}, we first introduce the notion of an alternative expression of the infinite product expansion of modular forms. For a modular form $f$, while its infinite product expansion is usually written as $f=q^{v_{\i}(f)}\prod_{n=1}^{\i}(1-q^{n})^{c(n)}$, Zagier introduced a modified infinite product expansion of modular form in order to study twisted traces of singular moduli of the elliptic modular $j$-function. To be more precise, let $D$ be a positive discriminant and $-d$ a negative discriminant. We assume that $D$ and $-d$ to be congruent to a square modulo $4N$. Let $\cQ_{d,N}:=\{[a,b,c]=ax^2+bxy+cy^2\in\cQ_{d}:a\equiv0\pmod{N}\}$, with the usual action of the congruence subgroup $\G_{0}(N)$ where $\cQ_{d}$ is the set of positive definite integral binary quadratic forms of discriminant $-d$. For each $Q\in\cQ_{d,N}$, we associate its unique root $\alpha_{Q}\in\H$. Then for $Q=[Na,b,c]\in\cQ_{dD,N}$, the genus character $\chi_{D}$ is defined as 
\begin{equation*}
\chi_{D}(Q):=
\begin{cases}
\big(\frac{\Delta}{n}\big), & \text{ if $\Delta|b^{2}-4Nac$ and $\frac{b^{2}-4Nac}{\Delta}$ is a square modulo $4N$ and $(a,b,c,\Delta)=1$},
\\
0, & \text{ otherwise},
\end{cases}
\end{equation*}
Here $n$ is any integer prime to $D$ represented by one of the quadratic forms $[N_{1}a,b,N_{2}c]$ with $N_{1}N_{2}=N$ and $N_{1},N_{2}>0$ (see \cite[\S 1.2]{GKZ} or \cite[\S 4]{BO}). Zagier \cite[Theorem 7]{Zagier} defined the rational function $\mathcal{H}_{D,d}(X)$ by
\begin{equation*}
\mathcal{H}_{D,d}(X)=\prod\limits_{Q\in\cQ_{Dd}/\G(1)}(X-j(\alpha_{Q}))^{\chi_{D}(Q)},
\end{equation*}
and showed that
\begin{equation*}
\mathcal{H}_{D,d}(j(\t))=\prod\limits_{n=1}^{\i}P_{D}(q^{n})^{A(n^{2}D,d)},
\end{equation*}
where
\begin{equation*}
P_{D}(t):=\exp\bigg(-\sqrt{D}\sum\limits_{r=1}^{\i}\Big(\frac{D}{r}\Big)\frac{t^{r}}{r}\bigg)=
\begin{cases}
\prod\limits_{0<n<D}(1-\zeta_{D}^{n}t)^{(\frac{D}{n})} &\; \text{ if $D>1$}
\\
1-t &\; \text{ if $D=1$}
\end{cases}
\in\Q(\sqrt{D})(t),
\end{equation*}
$\zeta_{D}:=e^{2\pi i/D}$, $(\frac{*}{*})$ is the Kronecker symbol, and $A(D,d)$ is the $D$th Fourier coefficient of $f_{d}=q^{-d}+O(q)$ in the space of weakly holomorphic modular forms of weight 1/2 on $\G_{0}(4)$ satisfying Kohnen plus condition.

\indent Later in order to describe the generalized Borcherds product explicitly, for instance \cite{BO,BS}, this alternative infinite product expansion of modular forms was employed. In a similar fashion, we adopt Zagier's idea of expressing modular forms via this type of infinite product expansion and apply it to all meromorphic modular forms. Using the series $P_{D}$, we now state our first result, which resolves convergence issues that may arise from the infinite product expansion of meromorphic modular forms.

\begin{thm}\label{first}
Suppose $D$ is a positive fundamental discriminant. Let
\begin{equation*}
f(\t)=q^{h}(1+\sum\limits_{n=1}^{\i}a(n)q^{n})
\end{equation*}
be a meromorphic modular form of weight $k$ on $\G_{0}(N)$. Then the following assertions are true.
\begin{enumerate}
\item 
There are uniquely determined complex numbers $c(D,n)$ satisfying
\begin{equation}\label{cDn}
f(\t)=q^{h}\prod\limits_{n=1}^{\i}P_{D}(q^{n})^{c(D,n)}.
\end{equation}
Here, we let $a^{b}=\exp(b\Log(a))$ where $\Log$ denotes the principal branch of the complex logarithm.

\item For each $n\geq1$, $c(D,n)$ is computed recursively from the values of $a(n)$ via the relation
\begin{equation}\label{recur}
c(D,n)=-\frac{1}{\sqrt{D}}a(n)-\frac{1}{n\sqrt{D}}\bigg(\sum\limits_{\substack{d|n\\d\ne n}}d\sqrt{D}c(D,d)\Big(\frac{D}{n/d}\Big)+\sum\limits_{1\leq u<n}a(n-u)\sum\limits_{d|u}d\sqrt{D}c(D,d)\Big(\frac{D}{u/d}\Big)\bigg).
\end{equation}
Moreover, the value
\begin{equation}\label{Dindependence}
\sum\limits_{u|n}u\sqrt{D}c(D,u)\Big(\frac{D}{n/u}\Big)
\end{equation}
is independent of $D$.
\end{enumerate}
\end{thm}

\begin{ex}
Let $D=8$ and $d=3$. The set of representatives of $\cQ_{24}/\G(1)$ is given by $\{[1,0,6],[2,0,3]\}$. Meanwhile, we clearly have $P_{8}(t)=\frac{(1-\zeta_{8}t)(1-\zeta_{8}^{7}t)}{(1-\zeta_{8}^{3}t)(1-\zeta_{8}^{5}t)}$. Therefore, we have
\begin{align*}
\mathcal{H}_{8,3}(j(\t))&=\frac{j(\t)-j(\alpha_{[1,0,6]})}{j(\t)-j(\alpha_{[2,0,3]})}=\frac{j(\t)-2417472-1707264\sqrt{2}}{j(\t)-2417472+1707264\sqrt{2}}
\\
&=1-3414528\sqrt{2}q+(11659001462784-8251985424384\sqrt{2})q^{2}
\\
&+(56353270574302101504-39847797031793227776\sqrt{2})q^{3}+\cdots
\\
&=1+\sum\limits_{n=1}^{\i}a(n)q^{n}=\prod\limits_{n=1}^{\i}\bigg(\frac{1-\sqrt{2}q^{n}+q^{2n}}{1+\sqrt{2}q^{n}+q^{2n}}\bigg)^{c(8,n)}.
\end{align*}
In agreement with \eqref{recur}, we have
\begin{align*}
&c(8,1)=-\frac{a(1)}{\sqrt{8}}=1707264, \;\; c(8,2)=-\frac{a(2)}{\sqrt{8}}+\frac{a(1)^{2}}{2\sqrt{8}}=4125992712192, 
\\
&c(8,3)=-\frac{a(3)}{\sqrt{8}}-\frac{a(1)-3a(1)a(2)+a(1)^{3}}{3\sqrt{8}}=13288900691444361984, \cdots,
\end{align*}
which agrees with the results in \cite[Theorem 7]{Zagier}.
\end{ex}

Given another infinite product expansion of a modular form, the action of multiplicative Hecke operators yields not a simple formula for Fourier coefficients of a modular form but rather a natural formula for exponents in its infinite product expansion. It is therefore natural to ask how to compute the exponents in the infinite product expansion of $f|\T(p)$ when $f$ is expressed in the form \eqref{cDn}. The following is a generalization of \cite[Theorem 1.2]{KS}.

\begin{thm}\label{exponents}
Suppose $D$ is a positive fundamental discriminant. Suppose that 
\begin{equation*}
f(\t)=q^{h}\prod\limits_{n=1}^{\i}P_{D}(q^{n})^{c(D,n)}
\end{equation*}
is a meromorphic modular form of weight $k$ on $\G_{0}(N)$. Then the action of multiplicative Hecke operators on $\M(N)$ is given by
\begin{equation*}
f(\t)|\T(p)=q^{h\beta(p)}\prod\limits_{n=1}^{\i}P_{D}(q^{n})^{c_{p}(D,n)},
\end{equation*}
where
\begin{equation*}
\beta(p):=
\begin{cases}
\sigma(p) & \text{ if } p\nmid N,
\\
1 & \text{ if } p|N,
\end{cases}
\end{equation*}
and 
\begin{equation}\label{cpDn}
c_{p}(D,n):=
\begin{cases}
c(D,\frac{n}{p})+pc(D,pn)+\big(\frac{D}{p}\big)\chi_{p}(n)c(D,n) & \text{ if } p\nmid N,
\\
pc(D,pn)+\big(\frac{D}{p}\big)\chi_{p}(n)c(D,n) & \text{ if } p|N.
\end{cases}
\end{equation}
Here, $\sigma(n)=\sum_{d|n}d$ is the sum of the divisors of $n$ and $\chi_{p}$ denotes the trivial Dirichlet character modulo $p$. Moreover, we have
\begin{equation}\label{cpr}
c_{p^{t}}(D,n)=c_{p^{t-1}}(D,n/p)+pc_{p^{t-1}}(D,pn)+\bigg(\frac{D}{p}\bigg)\chi_{p}(n)c_{p^{t-1}}(D,n)-pc_{p^{t-2}}(D,n),
\end{equation}
and
\begin{equation}\label{crs}
c_{rs}(D,n)=c_{r}(D,n)c_{s}(D,n) \;\; \text{ if } {\rm gcd}(r,s)=1.
\end{equation}
\end{thm}

\begin{ex}
Let $D=8$. Consider the modular form of weight 2 on $\G_{0}(11)$
\begin{align*}
f(\t)&=-\frac{1}{10}(E_{2}(\t)-11E_{2}(11\t)+24\eta(\t)^{2}\eta(11\t)^{2})=1+\sum\limits_{n=2}^{\i}a(n)q^{n}
\\
&=\prod\limits_{n=1}^{\i}\bigg(\frac{1-\sqrt{2}q^{n}+q^{2n}}{1+\sqrt{2}q^{n}+q^{2n}}\bigg)^{c(8,n)},
\end{align*}
where $E_{2}(\t)=1-24\sum_{n=1}^{\i}\sigma(n)q^{n}$ is the Eisenstein series of weight 2 and $\eta(\t)$ is the Dedekind eta function. For example, we can compute $f|\T(9)$ as follows: First, by using \eqref{recur}, we obtain $c(D,n)$ from $a(n)$'s. Then, using  \eqref{cpDn} yields
\begin{align*}
(f|\T(3))(\t)&=\bigg(\frac{1-\sqrt{2}q+q^{2}}{1+\sqrt{2}q+q^{2}}\bigg)^{-9\sqrt{2}}\bigg(\frac{1-\sqrt{2}q^{2}+q^{4}}{1+\sqrt{2}q^{2}+q^{4}}\bigg)^{-288\sqrt{2}}
\\
&\times\bigg(\frac{1-\sqrt{2}q^{3}+q^{6}}{1+\sqrt{2}q^{3}+q^{6}}\bigg)^{11742\sqrt{2}}\cdots.
\end{align*}
Finally, from \eqref{cpr}, we obtain
\begin{align*}
(f|\T(9))(\t)&=\bigg(\frac{1-\sqrt{2}q+q^{2}}{1+\sqrt{2}q+q^{2}}\bigg)^{35235\sqrt{2}}\bigg(\frac{1-\sqrt{2}q^{2}+q^{4}}{1+\sqrt{2}q^{2}+q^{4}}\bigg)^{1134001917\sqrt{2}}
\\
&\times\bigg(\frac{1-\sqrt{2}q^{3}+q^{6}}{1+\sqrt{2}q^{3}+q^{6}}\bigg)^{43213358093067\sqrt{2}}\cdots.
\end{align*} 
\end{ex}

\subsection{Divisor of modular forms and multiplicative Hecke operators}
The exponents in the infinite product expansion of a meromorphic modular form are closely related to its divisors. In \cite{BKO}, it was shown that for a given meromorphic modular form $f$, the values of the Faber polynomials $J_{m}(\t)$ at the zeros or poles of $f$ in the fundamental domain of $\SL(\Z)$ can be expressed as a $\Q$-linear combination of the exponents in the infinite product expansion of $f$ and the sum of divisors functions. Throughtout this article, by the \textit{divisor formula} we mean the relation given in \cite[Theorem 5]{BKO}. To state our next results, which describe how multiplicative Hecke operators fit into the framework of divisors of modular forms, we first introduce some auxiliary notions. At first, we brifely recall the definition of the modular functions $j_{N,n}(\t)$ on $\G_{0}(N)$, which will be used later. For $s\in\C$ and $m\geq0$, we define
\begin{equation*}
\phi_{m}(v,s):=
\begin{cases}
2\pi\sqrt{mv}I_{s-1/2}(2\pi mv) & \text{ if $m>0$},
\\
v^{s} & \text{ if $m=0$},
\end{cases}
\end{equation*}
where $I_{s}$ is the $I$-Bessel function. For ${\rm Re}(s)>1$, the \textit{Niebur-Poincar\'e series} $F_{N,-m}(\t,s)$ is defined by
\begin{equation*}
F_{N,-m}(\t,s):=\sum\limits_{\g\in\G_{0}(N)_{i\i}\backslash\G_{0}(N)}\phi_{m}(v,s)q^{-mu}|_{0}\g,
\end{equation*}
where $\G_{0}(N)_{i\i}=\{\pm\begin{psmallmatrix}1&n\\0&1\end{psmallmatrix}:n\in\Z\}$. Various properties of $F_{N,-m}$ have been discussed in \cite{Nie}. The function $j_{N,n}(\t)$ is defined by the coefficient in the Laurent expansion of the Niebur-Poincar\'e series at $s=1$. This is a standard example of a \textit{polyharmonic Maass form}.
Next, for each function $f:X_{0}(N)\rightarrow\C$, we define a map $\mathcal{D}_{f}:{\rm Div}(X_{0}(N))\rightarrow\C$ by
\begin{equation*}
\mathcal{D}_{f}(D):=\sum\limits_{[z]\in X_{0}(N)}n_{z}f(z)
\end{equation*}
for $D=\sum_{[z]\in X_{0}(N)}n_{z}[z]$. The quantity $\mathcal{D}_{f}(\div(g))$ is called the \textit{Rohrlich-type divisor sum}. For further details, we refer to \cite{BK,JKKMd,JKKMu,Roh}.

\begin{thm}\label{Rohrlich}
Let
\begin{equation*}
f(\t)=q^{h}\prod\limits_{n=1}^{\i}P_{D}(q^{n})^{c(D,n)}
\end{equation*}
be a meromorphic modular form of weight $k$ on $\G_{0}(N)$. Then the following are true:
\begin{enumerate}
\item we have
\begin{equation*}
\sum\limits_{u|n}u\sqrt{D}c(D,u)\bigg(\frac{D}{n/u}\bigg)=\mathcal{D}_{j_{N,n}}(\div f).
\end{equation*}
\item For $p\nmid N$, we have
\begin{equation*}
\sum\limits_{u|n}u\sqrt{D}c_{p^{r}}(D,u)\bigg(\frac{D}{n/u}\bigg)=\mathcal{D}_{j_{N,n}|T(p^{r})}(\div f).
\end{equation*}
In particular, if $n=1$, we have 
\begin{equation*}
\sqrt{D}c_{p^{r}}(D,1)=\mathcal{D}_{j_{N,p^{r}}}(\div f).
\end{equation*}
\end{enumerate}
\end{thm}

\begin{rmk}
\begin{enumerate}
\item If we put $N=D=1$ in Theorem \ref{Rohrlich}-(1), then the divisor formula is recovered.
\item Recently Jeon, Kang, the first author, and Matsusaka \cite{JKKMd} studied the Hecke equivariance of the divisor map. They regarded the divisor formula as a pairing $(*,*)_{BKO}$ between meromorphic modular forms and modular functions. They proved that a multiplicative Hecke operator is an adjoint operator of the usual Hecke operator with respect to this pairing. That is, they obtained 
\begin{equation}\label{pairing}
(j_{1}|T(n),f)_{BKO}=(j_{1},f|\T(n))_{BKO}.
\end{equation}
We mention that, by using \eqref{cpDn} with $D=1$, one can see that $c_{m}(1)=\sum_{u|m}c(u)$. This relation combined with the divisor formula leads to \eqref{pairing} directly. We also note that $c_{m}(1)=\sum_{u|m}c(u)$ can be derived from Theorem \ref{Rohrlich}. That is, if we put $N=D=1$, $n=m$ in Theorem \ref{Rohrlich}-(1) and $N=D=n=1$, $p^{r}=m$ in Theorem \ref{Rohrlich}-(2), then we get $c_{m}(1)=\sum_{u|m}c(u)$. 
\end{enumerate}
\end{rmk}

\subsection{Twisted traces of singular moduli}

In this subsection we suppose $N$ is a positive integer such that the genus of $\G_{0}(N)$ is zero or one. Furthermore, we set $D$ is a positive fundamental discriminant, $\beta\in\Z$ such that $D\equiv\beta^{2}\pmod{4N}$ and $-d$ is a negative integer congruent to a square modulo $4N$. For a fixed $\beta\pmod{2N}$ such that $\beta^{2}\equiv-d\pmod{4N}$, we define
\begin{equation*}
\cQ_{d,N,\beta}=\{Q=[a,b,c]:a\equiv0\pmod{N},b^{2}-4ac=-d,b\equiv\beta\pmod{2N}\}\subset\cQ_{d,N}.
\end{equation*}
Note that $\G_{0}(N)$ acts on the set $\cQ_{d,N,\beta}$. We now define the \textit{twisted trace of singular moduli} associated with a weakly holomorphic modular function $f$ on $\G_{0}(N)$ by
\begin{equation*}
\Tr_{D,d}(f):=\sum\limits_{Q\in\cQ_{Dd,N}/\G_{0}(N)}\frac{\chi_{D}(Q)}{\omega_{Q}}f(\alpha_{Q}) \;\; \text{ where } \omega_{Q}=|\G_{0}(N)_{Q}/\{\pm I\}|.
\end{equation*}

The congruence of twisted traces of singular moduli for modular functions has been studied by many authors, for instance, see \cite{Ahl,Guerzhoy,JKK}. In this subsection, we add new congruences using the Hecke equivariance of multiplicative Hecke operators.

\begin{thm}\label{twistedtraces}
The following are true:
\begin{enumerate}
\item We have
\begin{equation*}
\frac{1}{\sqrt{D}}Tr_{D,d}(j_{N,n})\in\Z
\end{equation*}
\item For $p\nmid N$, we have
\begin{equation*}
\frac{1}{\sqrt{D}}Tr_{D,d}(j_{N,p})\equiv\bigg(\frac{D}{p}\bigg)\bigg(\frac{1}{\sqrt{D}}Tr_{D,d}(j_{N,1})\bigg)\pmod{p}
\end{equation*}
In particular, if genus of $\G_{0}(N)$ is one and $\Tr_{D,d}(1)=0$ then
\begin{equation*}
\frac{1}{\sqrt{D}}Tr_{D,d}(f_{N,p})\equiv\bigg(\bigg(\frac{D}{p}\bigg)+\alpha_{p}\bigg)Tr_{D,d}(j_{N,1})\pmod{p}
\end{equation*}
where $f_{N,n}=j_{N,n}+\alpha_{n}j_{N,1}$ is the canonical basis element in $M_{0}^{\#}(N)$ ($\alpha_{n}\in\C$). Here $M_{0}^{\#}(N)$ denotes the space of weakly holomorphic modular functions on $\G_{0}(N)$ whose poles are supported only at the cusp $i\i$.
\item For $p\nmid N$, we have
\begin{equation*}
\frac{1}{\sqrt{D}}Tr_{D,d}(j_{N,p^{r}})=\sum\limits_{t=0}^{r}\bigg(\frac{-d}{p}\bigg)^{r-t}\frac{1}{\sqrt{D}}Tr_{D,p^{2t}d}(j_{N,1})
\end{equation*}
\end{enumerate}
\end{thm}

\begin{ex}
\begin{enumerate}
\item Let $N=7$, $D=13$, $d=4$, and $p=3$. The Hauptmodul for $\G_{0}(7)$ is given by
\begin{equation*}
f_{7,1}(\t)=\bigg(\frac{\eta(\t)}{\eta(7\t)}\bigg)^{4}+4.
\end{equation*}
In this case, one has $f_{7,3}(\t)=f_{7,1}(\t)^{3}-6f_{7,1}(\t)-24$, 
\begin{align*}
&\cQ_{52,7,2}/\G_{0}(7)=\{[14,2,1],[7,2,2]\}, \;\;  \cQ_{52,7,-2}/\G_{0}(7)=\{[49,12,1],[7,-2,2]\},
\\
&\cQ_{468,7,6}/\G_{0}(7)=\{[126,6,1],[63,6,2],[42,6,3],[21,6,6],[7,6,18],[154,62,7],
\\
&[238,90,9],[77,48,9],[14,6,9],[98,62,11]\} \text{ and } 
\\
&\cQ_{468,7,-6}/\G_{0}(7)=\{[133,8,1],[119,22,2],[147,36,3],[21,-6,6],[63,36,7],
\\
&[7,-6,18],[49,36,9],[14,-6,9],[182,78,9],[266,106,11]\}.
\end{align*}
Note that $j_{7,n}=f_{7,n}+c$ for some constant $c$ and therefore their twisted traces of singular moduli are equal. By using these data we compute
\begin{equation*}
\frac{1}{\sqrt{13}}\Tr_{13,4}(j_{7,3})=8244, \;\;\; \frac{1}{\sqrt{13}}\Tr_{13,4}(j_{7,1})=-6, \;\;\; \frac{1}{\sqrt{13}}\Tr_{13,36}(j_{7,1})=8238,
\end{equation*}
which illustrate Theorem \ref{twistedtraces}-(2) and (3).
\item Let $N=11$, $D=5$, $d=7$, and $p=3$. Note that the genus of $\G_{0}(11)$ is one. In this case, we have $f_{11,3}=\sum_{\g\in\G_{0}(11)/\G}\frac{G_{5}^{4}}{G_{1}^{3}G_{3}}+1$ where $G_{k}$ is generalized Dedekind eta-functions and $\G$ is the intermediate subgroup between $\G_{1}(11)$ and $\G_{0}(11)$ with $[\G_{0}(11):\G]=5$(cf, see the table in \cite[\S 4.1]{Yang}). Furthermore one has $\cQ_{35,11,3}/\G_{0}(11)=\{[11.3.1],[187,47,3]\}$ and $\cQ_{35,11,-3}/\G_{0}(11)=\{[99,19,1],[143,41,3]\}$, and $\alpha_{3}=1$. From these data, we have
\begin{equation*}
\frac{1}{\sqrt{5}}\Tr_{5,7}(f_{11,3})=-120\equiv0\pmod{3}.
\end{equation*}
which illustrate Theorem \ref{twistedtraces}-(2).
\end{enumerate}
\end{ex}

\subsection{Multiplicative Hecke eigenform with nonintegral Fourier coefficients}

The second purpose of this article is to investigate the multiplicative Hecke eigenform $f$ when $f\not\in\Z((q))$. In our previous work \cite[Theorem 1.14]{KS}, we proved that if $f$ is a meromorphic modular form on $\G_{0}(N)$ and either $N$ is squarefree or $f\in\Z((q))$ then the following are equivalent:
\begin{enumerate}[(i)]
\item $f$ is an eta quotient.
\item $f$ is a multiplicative Hecke eigenform.
\item $f$ has no poles or zeros in $\H$.
\end{enumerate}
It is noteworthy that an integer Fourier coefficient condition is necessary. Indeed if $f$ is such a form but its Fourier coefficients are nonintegral, then by the Drinfeld-Manin theorem \cite{Drin,Man} and \cite[Corollary 8]{RW} there exists a modular form that is not an eta-quotient. However, in \cite{KS}, it was proven that if $f$ is an eta quotient then it is a multiplicative Hecke eigenform, and if $f$ is a multiplicaitive Hecke eigenform, then it has no poles or zeros in $\H$. Therefore, a natural question arising from this observation is that if we eliminate the integer Fourier coefficient condition, then which of (\rnum{3})$\Rightarrow$(\rnum{2}) or (\rnum{2})$\Rightarrow$(\rnum{1}) does not hold? Here, we provide an answer to this question.

\begin{thm}\label{counter}
Suppose $N$ is a positive integer such that $p^{2}|N$ for some odd prime $p$. Then there exists a meromorphic modular form on $\G_{0}(N)$ such that it has no poles or zeros in $\H$ but is not a multiplicative Hecke eigenform.
\end{thm}

\begin{ex}
Let $N=9$. Then the Hauptmodul for $\G_{0}(9)$ is given by $\eta(\t)^{3}/\eta(9\t)^{3}$. From the transformation rule of the Dedekind eta function (for instance, see \cite[Theorem 5.8.1]{CS}), we have
\begin{equation*}
\alpha:=\lim_{\t\rightarrow i\i}\frac{\eta(\t)^{3}}{\eta(9\t)^{3}}\bigg|_{0}\begin{pmatrix}1&0\\3&1\end{pmatrix}=-\frac{9+3\sqrt{3}i}{2}.
\end{equation*}
It follows that if we put $f(\t):=\eta(\t)^{3}/\eta(9\t)^{3}-\alpha$, then $f$ has a simple pole at $i\i$ and a simple zero at the cusp 1/3. Thus, it has no zeros or poles in $\H$, and clearly has nonintegral Fourier coefficients. However, we see from \eqref{recur} that
\begin{equation*}
f(\t)=q^{-1}\prod\limits_{n=1}^{\i}(1-q^{n})^{c(n)}=q^{-1}(1-q)^{-(3+3\sqrt{3}i)/2}(1-q^{2})^{(-3+6\sqrt{3}i)/2}(1-q^{3})^{(9+\sqrt{3}i)/2}\cdots,
\end{equation*}
and from Theorem \ref{exponents} (or \cite[Theorem 1.2]{KS}), that
\begin{equation*}
f|\T(2)=q^{-3}(1-q)^{(-9+9\sqrt{3}i)/2}(1-q^{2})^{(-9-18\sqrt{3}i)/2}(1-q^{3})^{(27-3\sqrt{3}i)/2}\cdots,
\end{equation*}
which is not equal to $f(\t)^{3}$. Therefore, we conclude that $f$ is not a multiplicative Hecke eigenform.
\end{ex}

\begin{rmk}
Theorem \ref{counter} does not imply that \textup{(\rnum{2})}$\Rightarrow$\textup{(\rnum{1})} is true. Therefore, it would be interesting to determine whether \textup{(\rnum{2})} implies \textup{(\rnum{1})} or not when $f\not\in\Z((q))$.
\end{rmk}

The rest of this article is organized as follows. In Section 2, we prove Theorems \ref{first} and \ref{exponents}. In Section 3, we prove Theorem \ref{Rohrlich}. In section 4, we prove Theorem \ref{twistedtraces}. In section 5, we prove Theorem \ref{counter}.

\section{Proofs of Theorems \ref{first} and \ref{exponents}}

\begin{proof}[Proof of Theorem \ref{first}]
\begin{enumerate}
\item
For convenience, let $F(q):=f(\t)$. Then, the function $qF'(q)/F(q)$ is holomorphic at $q=0$. Write its Taylor expansion as
\begin{equation*}
\frac{qF'(q)}{F(q)}=h-\sum\limits_{n=1}^{\i}\alpha(n)q^{n}.
\end{equation*}
Fix $D>0$. For each $n\geq1$ we define $c(D,n)$ recursively using
\begin{equation*}
\alpha(n)=\sum\limits_{0<m<D}\sum\limits_{d|n}\bigg(\frac{D}{m}\bigg)\zeta_{D}^{m\frac{n}{d}}dc(D,d).
\end{equation*}
We can see that
\begin{align*}
\frac{d}{dq}\log(F(q)q^{-h})&=\frac{F'(q)}{F(q)}-\frac{h}{q}
\\
&=-\sum\limits_{0<m<D}\sum\limits_{n=1}^{\i}\sum\limits_{d|n}\bigg(\frac{D}{m}\bigg)\zeta_{D}^{m\frac{n}{d}}dc(D,d)\frac{d}{dq}\bigg(\frac{q^{n}}{n}\bigg)
\\
&=-\sum\limits_{0<m<D}\sum\limits_{t=1}^{\i}\bigg(\frac{D}{m}\bigg)c(D,t)\frac{d}{dq}\bigg(\sum\limits_{r=1}^{\i}\frac{\zeta_{D}^{mr}q^{tr}}{r}\bigg)
\\
&=\frac{d}{dq}\bigg(\sum\limits_{n=1}^{\i}\sum\limits_{0<m<D}\bigg(\frac{D}{m}\bigg)c(D,n)\log(1-\zeta_{D}^{m}q^{n})\bigg).
\end{align*}
We thus obtain
\begin{equation*}
\log(F(q)q^{-h})=\sum\limits_{n=1}^{\i}\sum\limits_{0<m<D}\bigg(\frac{D}{m}\bigg)c(D,n)\log(1-\zeta_{D}^{m}q^{n}).
\end{equation*}
The values $\sum_{0<m<D}\big(\frac{D}{m}\big)c(D,n)\log(1-\zeta_{D}^{m}q^{n})$ and $\sum_{0<m<D}\log(1-\zeta_{D}^{m}q^{n})^{(\frac{D}{m})c(D,n)}$ differ by integer multiples of $2\pi i$. Since $\big(\frac{D}{m}\big)c(D,n)\log(1-\zeta_{D}^{m}q^{n})\rightarrow0$ as $n\rightarrow\i$, the same is true for $\log(1-\zeta_{D}^{m}q^{n})^{(\frac{D}{m})c(D,n)}$ for fixed $0<m<D$. Hence we see that there exist an integer $M$ such that 
\begin{equation*}
\log(F(q)q^{-h})\sum\limits_{n=1}^{\i}\sum\limits_{0<m<D}\log(1-\zeta_{D}^{m}q^{n})^{(\frac{D}{m})c(D,n)}+2\pi iM.
\end{equation*}
Taking the exponential on both sides yields
\begin{equation*}
f(\t)q^{-h}=F(q)q^{-h}=\prod\limits_{n=1}^{\i}\prod\limits_{0<m<D}(1-\zeta_{D}^{m}q^{n})^{(\frac{D}{m})c(D,n)}.
\end{equation*}
\item
We have
\begin{align*}
f(\t)&=\prod\limits_{n=1}^{\i}\prod\limits_{0<m<D}(1-\zeta_{D}^{m}q^{n})^{(\frac{D}{m})c(D,n)}=\prod\limits_{n=1}^{\i}\prod\limits_{0<m<D}\exp\Big(\log(1-\zeta_{D}^{m}q^{n})^{(\frac{D}{m})c(D,n)}\Big)
\\
&=\exp\bigg(\sum\limits_{n=1}^{\i}\sum\limits_{0<m<D}\bigg(\frac{D}{m}\bigg)c(D,n)\log(1-\zeta_{D}^{m}q^{n})\bigg)
\\
&=\exp\bigg(-\sum\limits_{0<m<D}\sum\limits_{n=1}^{\i}\sum\limits_{r=1}^{\i}\bigg(\frac{D}{m}\bigg)c(D,n)\frac{(\zeta_{D}^{m}q^{n})^{r}}{r}\bigg)
\\
&=\exp\bigg(-\sum\limits_{0<m<D}\sum\limits_{n=1}^{\i}\sum\limits_{d|n}\bigg(\frac{D}{m}\bigg)dc(D,d)\frac{\zeta_{D}^{m\frac{n}{d}}q^{n}}{n}\bigg).
\end{align*}
Let $\Theta:=\frac{1}{2\pi i}\frac{d}{d\t}$. Since $\Theta(\log(f))=\Theta(f)/f$ and $\Theta(\sum_{n=n_{0}}a(n)q^{n})=\sum_{n=n_{0}}na(n)q^{n}$, we deduce that
\begin{equation*}
\bigg(1+\sum\limits_{n=1}^{\i}a(n)q^{n}\bigg)\bigg(\sum\limits_{0<m<D}\sum\limits_{n=1}^{\i}\sum\limits_{d|n}\bigg(\frac{D}{m}\bigg)dc(D,d)\zeta_{D}^{m\frac{n}{d}}q^{n}\bigg)=-\sum\limits_{n=1}^{\i}na(n)q^{n}.
\end{equation*}
Comparing the $q^{n}$ coefficients on both sides we obtain
\begin{equation}\label{nan}
-na(n)=\sum\limits_{0<m<D}\sum\limits_{d|n}\bigg(\frac{D}{m}\bigg)dc(D,d)\zeta_{D}^{m\frac{n}{d}}+\sum\limits_{1\leq u<n}a(n-u)\bigg(\sum\limits_{0<m<D}\sum\limits_{d|u}\bigg(\frac{D}{m}\bigg)dc(D,d)\zeta_{D}^{m\frac{u}{d}}\bigg).
\end{equation}
On the other hand, since $\sqrt{D}(\frac{D}{r})=\sum_{0<m<D}(\frac{D}{m})\zeta_{D}^{mr}$, we have
\begin{equation}\label{nan2}
\sum\limits_{0<m<D}\sum\limits_{d|n}\bigg(\frac{D}{m}\bigg)dc(D,d)\zeta_{D}^{m\frac{n}{d}}=\sum\limits_{d|n}d\sqrt{D}c(D,d)\bigg(\frac{D}{n/d}\bigg).
\end{equation}
Substituting \eqref{nan2} into \eqref{nan} yields that
\begin{equation}\label{nan3}
-na(n)=\sum\limits_{d|n}d\sqrt{D}c(D,d)\bigg(\frac{D}{n/d}\bigg)+\sum\limits_{1\leq u<n}a(n-u)\sum\limits_{d|u}d\sqrt{D}c(D,d)\bigg(\frac{D}{u/d}\bigg).
\end{equation}
Moreover, one can see that 
\begin{equation*}
\sum\limits_{d|n}d\sqrt{D}c(D,d)\bigg(\frac{D}{n/d}\bigg)
\end{equation*}
is independent of $D$ by using induction on $n$ in \eqref{nan3}. This proves the second assertion.
\end{enumerate}
\end{proof}

\begin{proof}[Proof of Theorem \ref{exponents}]
We first assume that $p\nmid N$. From the definition of the multiplicative Hecke operator, we have
\begin{align*}
f|\T(p)&=\varepsilon f(p\t)\prod\limits_{j=0}^{p-1}f\Big(\frac{\t+j}{p}\Big)
\\
&=\varepsilon q^{ph}\prod\limits_{n=1}^{\i}\prod\limits_{0<m<D}(1-\zeta_{D}^{m}q^{pn})^{(\frac{D}{m})c(D,n)}\prod\limits_{j=0}^{p-1}\zeta_{p}^{jh}q^{\frac{h}{p}}\prod\limits_{n=1}^{\i}\prod\limits_{0<m<D}(1-\zeta_{D}^{m}\zeta_{p}^{jn}q^{\frac{n}{p}})^{(\frac{D}{m})c(D,n)}
\\
&=q^{h(p+1)}\prod\limits_{n=1}^{\i}\prod\limits_{0<m<D}(1-\zeta_{D}^{m}q^{pn})^{(\frac{D}{m})c(D,n)}\prod\limits_{\substack{n=1\\p|n}}^{\i}\prod\limits_{0<m<D}(1-\zeta_{D}^{m}q^{\frac{n}{p}})^{(\frac{D}{m})pc(D,n)}
\\
&\times\prod\limits_{\substack{n=1\\(n,p)=1}}^{\i}\prod\limits_{j=0}^{p-1}\prod\limits_{0<m<D}(1-\zeta_{D}^{m}\zeta_{p}^{jn}q^{\frac{n}{p}})^{(\frac{D}{m})c(D,n)}
\\
&=q^{h(p+1)}\prod\limits_{n=1}^{\i}\prod\limits_{0<m<D}(1-\zeta_{D}^{m}q^{pn})^{(\frac{D}{m})c(D,n)}\prod\limits_{n=1}^{\i}\prod\limits_{0<m<D}(1-\zeta_{D}^{m}q^{n})^{(\frac{D}{m})pc(D,pn)}
\\
&\times\prod\limits_{\substack{n=1\\(n,p)=1}}^{\i}\prod\limits_{0<m<D}(1-\zeta_{D}^{pm}q^{n})^{(\frac{D}{m})c(D,n)}.
\end{align*}
The last equality follows from the fact that $\prod_{j=0}^{p-1}(1-\zeta_{p}^{j}X)=1-X^{p}$. Furthermore, we assume that $(p,D)=1$. Then we have
\begin{align*}
f|\T(p)&=q^{h(p+1)}\prod\limits_{n=1}^{\i}\prod\limits_{0<m<D}(1-\zeta_{D}^{m}q^{pn})^{(\frac{D}{m})c(D,n)}\prod\limits_{n=1}^{\i}\prod\limits_{0<m<D}(1-\zeta_{D}^{m}q^{n})^{(\frac{D}{m})pc(D,pn)}
\\
&\times\prod\limits_{\substack{n=1\\(n,p)=1}}^{\i}\prod\limits_{0<m<D}(1-\zeta_{D}^{m}q^{n})^{(\frac{D}{m\bar{p}})c(D,n)}
\\
&=q^{h(p+1)}\prod\limits_{n=1}^{\i}\prod\limits_{0<m<D}(1-\zeta_{D}^{m}q^{n})^{(\frac{D}{m})c(D,\frac{n}{p})+(\frac{D}{m})pc(D,pn)+(\frac{D}{m\bar{p}})\chi_{p}(n)c(D,n)}
\\
&=q^{h(p+1)}\prod\limits_{n=1}^{\i}\prod\limits_{0<m<D}(1-\zeta_{D}^{m}q^{n})^{(\frac{D}{m})\big(c(D,\frac{n}{p})+pc(D,pn)+(\frac{D}{\bar{p}})\chi_{p}(n)c(D,n)\big)}.
\end{align*}
Here $\bar{p}$ denotes the inverse of $p$ modulo $D$. Then the claim follows from the fact that $(\frac{D}{p})=(\frac{D}{\bar{p}})$. Next, we assume that $(p,D)=p$. In this case, we have
\begin{align*}
f|\T(p)&=q^{h(p+1)}\prod\limits_{n=1}^{\i}\prod\limits_{0<m<D}(1-\zeta_{D}^{m}q^{pn})^{(\frac{D}{m})c(D,n)}\prod\limits_{n=1}^{\i}\prod\limits_{0<m<D}(1-\zeta_{D}^{m}q^{n})^{(\frac{D}{m})pc(D,pn)}
\\
&\times\prod\limits_{\substack{n=1\\(n,p)=1}}^{\i}\prod\limits_{\substack{0<m<D\\p|m}}(1-\zeta_{D}^{m}q^{n})^{\big((\frac{D}{m/p})+(\frac{D}{m/p+D/p})+\cdots+(\frac{D}{m/p+(p-1)D/p})\big)c(D,n)}
\\
&=q^{h(p+1)}\prod\limits_{n=1}^{\i}\prod\limits_{0<m<D}(1-\zeta_{D}^{m}q^{n})^{(\frac{D}{m})\big(c(D,\frac{n}{p})+pc(D,pn)\big)}.
\end{align*}
The next lemma shows that the last equality holds. Next, we assume that $p|N$. In this case, $f|\T(p)$ is expressed as
\begin{align*}
f|\T(p)&=\varepsilon\prod\limits_{j=0}^{p-1}f\Big(\frac{\t+j}{p}\Big)=\varepsilon\prod\limits_{j=0}^{p-1}\zeta_{p}^{jh}q^{\frac{h}{p}}\prod\limits_{n=1}^{\i}\prod\limits_{0<m<D}(1-\zeta_{D}^{m}\zeta_{p}^{jn}q^{\frac{n}{p}})^{(\frac{D}{m})c(D,n)}
\\
&=q^{h}\prod\limits_{\substack{n=1\\p|n}}^{\i}\prod\limits_{0<m<D}(1-\zeta_{D}^{m}q^{\frac{n}{p}})^{(\frac{D}{m})pc(D,n)}\prod\limits_{\substack{n=1\\(n,p)=1}}^{\i}\prod\limits_{j=0}^{p-1}\prod\limits_{0<m<D}(1-\zeta_{D}^{m}\zeta_{p}^{jn}q^{\frac{n}{p}})^{(\frac{D}{m})c(D,n)}.
\end{align*}
From the same computation, we deduce that
\begin{equation*}
f|\T(p)=q^{h}\prod\limits_{n=1}^{\i}\prod\limits_{0<m<D}(1-\zeta_{D}^{m}q^{n})^{(\frac{D}{m})\big(pc(D,pn)+(\frac{D}{p})\chi_{p}(n)c(D,n)\big)}.
\end{equation*}
Next, we prove equations \eqref{cpr} and \eqref{crs}. We see from \cite[Theorem 1.7]{KS} that
\begin{equation*}
f|\T(p^{t-1})\T(p)=f|\T(p^{t})\cdot \Big(f|\T(p^{t-2})\Big)^{p}.
\end{equation*}
Therefore, we have
\begin{align*}
c_{p^{t}}(n)&=c_{p^{t-1},p}(n)-pc_{p^{t-2}}(n)
\\
&=c_{p^{t-1}}(D,n/p)+pc_{p^{t-1}}(D,pn)+\bigg(\frac{D}{p}\bigg)\chi_{p}(n)c_{p^{t-1}}(D,n)-pc_{p^{t-2}}(D,n).
\end{align*}
Here, $c_{p^{t-1},p}(n)$ denotes the $n$th exponent of $f|\T(p^{t-1})\T(p)=f|\T(p)\T(p^{t-1})$. The proof of \eqref{crs} is the same as that of \eqref{cpr}.
\end{proof}

\begin{lemma}
Let $D$ be a positive fundamental discriminant. Let $p$ be a prime such that $p|D$ and $m$ be a positive integer such that $p|m$ and $0<m<D$. Then the value
\begin{equation}\label{jacobi}
\Big(\frac{D}{m/p}\Big)+\Big(\frac{D}{m/p+D/p}\Big)+\cdots+\Big(\frac{D}{m/p+(p-1)D/p}\Big)
\end{equation}
is equal to zero.
\end{lemma}

\begin{proof}
For convenience we let $D':=D/p$. We first assume that $D$ is congruent to 1 modulo 4 and squarefree. In this case, we have
\begin{equation*}
\Big(\frac{D}{m/p+iD/p}\Big)=\Big(\frac{m/p+iD/p}{D}\Big)=\Big(\frac{m/p+iD'}{p}\Big)\Big(\frac{m/p+iD'}{D'}\Big)
\end{equation*}
for all $0\leq i<p$. Since $p$ is copirme to $D'$, the set $\{m/p+iD':0\leq i<p\}$ forms a complete residue system modulo $p$. Therefore, we have
\begin{equation*}
\sum\limits_{i=0}^{p-1}\Big(\frac{D}{m/p+iD/p}\Big)=\Big(\frac{m/p}{D'}\Big)\sum\limits_{i=0}^{p-1}\Big(\frac{i}{p}\Big)=0.
\end{equation*}
Next, we assume that $D=4l$ where $l\equiv3\pmod{4}$, $l$ is squarefree and $p=2$. Since $D'\equiv2\pmod{4}$, we have
\begin{equation*}
\Big(\frac{D}{m/2}\Big)+\Big(\frac{D}{m/2+D'}\Big)=\pm\bigg(\Big(\frac{m/2}{D}\Big)-\Big(\frac{m/2+D'}{D}\Big)\bigg)=\pm\Big(\frac{m/2}{D'/2}\Big)\bigg(\Big(\frac{m/2}{4}\Big)-\Big(\frac{m/2+D'}{4}\Big)\bigg)=0.
\end{equation*}
Similarly we assume that $D=4l$ where $l\equiv2\pmod{4}$, $l$ is squarefree and $p=2$. Since $D'\equiv0\pmod{4}$, we have
\begin{equation*}
\Big(\frac{D}{m/2}\Big)+\Big(\frac{D}{m/2+D'}\Big)=\pm\bigg(\Big(\frac{m/2}{D}\Big)+\Big(\frac{m/2+D'}{D}\Big)\bigg)=\pm\Big(\frac{m/2}{D'/4}\Big)\bigg(\Big(\frac{m/2}{8}\Big)+\Big(\frac{m/2+D'}{8}\Big)\bigg)=0.
\end{equation*}
The last equality follows from the congruence $D'\equiv4\pmod{8}$. Finally we assume that $D=4l$ where $l\equiv2,3\pmod{4}$, $l$ is squarefree and $p>2$. In this case, $m/p+iD/p$ is congruent modulo 4 for all $0\leq i<p$, since $D'\equiv0\pmod{4}$. Thus, \eqref{jacobi} is
\begin{equation*}
\sum\limits_{i=0}^{p-1}\Big(\frac{D}{m/p+iD'}\Big)=\pm\sum\limits_{i=0}^{p-1}\Big(\frac{m/p+iD'}{D}\Big)=\pm\Big(\frac{m/p}{D'}\Big)\sum\limits_{i=0}^{p-1}\Big(\frac{m/p+iD'}{p}\Big)=0
\end{equation*}
since $p$ is coprime to $D'$. 
\end{proof}

\section{Proof of Theorem \ref{Rohrlich}}

\begin{proof}
We first note that
\begin{align*}
\sum\limits_{u|n}u\sqrt{D}c(D,u)\bigg(\frac{D}{n/u}\bigg)\overset{\eqref{Dindependence}}{=}&\sum\limits_{u|n}uc(1,u)=-{\rm Coeff}_{q^{n}}\bigg(\frac{\Theta f}{f}\bigg)
\\
\overset{\text{\cite[Corollary 4.5]{JKKMu}}}{=}&\sum\limits_{z\in\G_{0}(N)\backslash\H}\frac{\ord_{z}(f)}{\omega_{z}}j_{N,n}(z)+\sum\limits_{\rho\in C_{N}}\ord_{\rho}(f)j_{N,n}(\rho).
\end{align*}
Here, ${\rm Coeff}_{q^{n}}(f)$ denotes the $n$th Fourier coefficient of $f(\t)$.
Therefore, by the definition of Rohrlich-type divisor sum, we have
\begin{align*}
\sum\limits_{u|n}u\sqrt{D}c(D,u)\bigg(\frac{D}{n/u}\bigg)=\mathcal{D}_{j_{N,n}}(\div f).
\end{align*}
This is the proof of the first assertion. Next, for $p\nmid N$, according to  \cite[Corollary 1.3]{JKKMd}, we obtain 
\begin{align*}
\mathcal{D}_{j_{N,n}|T(p^{r})}(\div f)=\mathcal{D}_{j_{N,n}}(\div(f|\T(p^{r})).
\end{align*}
Therefore, by the first assertion, we obtain
\begin{equation*}
\sum\limits_{u|n}u\sqrt{D}c_{p^{r}}(D,u)\bigg(\frac{D}{n/u}\bigg)=\mathcal{D}_{j_{N,n}|T(p^{r})}(\div f).
\end{equation*}
In particular, if $n=1$, we obtain
\begin{equation*}
\sqrt{D}c_{p^{r}}(D,1)=\mathcal{D}_{j_{N,p^{r}}}(\div f).
\end{equation*}
\end{proof}

\section{Proof of Theorem \ref{twistedtraces}}

\begin{proof}
Suppose that the genus of $\G_{0}^{*}(N)$ is zero. Let $f:=\Psi_{D}(f_{N,d})=\mathcal{H}_{D,d,N}(j_{N,1})$. Then, by \cite[Theorem 6.1]{BO}, we have $c(D,n)=A^{(N)}(n^{2}D,d)\in\Z$ where $A^{(N)}(D,d)$ denotes the level $N$ analogue of $A(D,d)$. Note that we have $\Tr_{D,d}(f)=\mathcal{D}_{j_{N,n}}(\div f)$. Thus, by Theorem \ref{Rohrlich}, we immediately obtain $\frac{1}{\sqrt{D}}\Tr_{D,d}(j_{N,n})=\sum_{u|n}uc(D,u)\big(\frac{D}{n/u}\big)\in\Z$. Next, we suppose $p\nmid N$. By \cite[Theorem 1.1 and 1.2]{JKKMd}, we see that
\begin{equation}\label{Trcongruence}
\Tr_{D,d}(j_{N,n}|T(p^{r}))=\sum\limits_{u|n}u\sqrt{D}c_{p^{r}}(D,u)\bigg(\frac{D}{n/u}\bigg).
\end{equation}
If $r=n=1$, then it reduces to $\Tr_{D,d}(j_{N,1}|T(p))=\sqrt{D}c_{p^{r}}(D,1)$. Since the trace of singular moduli of the constant function is zero, we see that the left-hand side is equal to $\Tr_{D,d}(j_{N,p})$. On the other hand, by \eqref{cpDn}, the right-hand side can be expressed as $\sqrt{D}(pc(D,p)+(\frac{D}{p})c(D,1))$. Hence, we get 
\begin{equation*}
\frac{1}{\sqrt{D}}\Tr_{D,d}(j_{N,p})\equiv\bigg(\frac{D}{p}\bigg)\frac{1}{\sqrt{D}}\Tr_{D,d}(j_{N,1})\pmod{p}.
\end{equation*}
In particular, if the genus of $\G_{0}(N)$ is one, then we can write $f_{N,m}=j_{N,m}+\alpha_{m}j_{N,1}$ for some constant $\alpha_{m}$. Therefore, by adding $\frac{1}{\sqrt{D}}\Tr_{D,d}(\alpha_{p}j_{N,1})$ to both sides, we obtain 
\begin{equation*}
\frac{1}{\sqrt{D}}\Tr_{D,d}(f_{N,p})\equiv\bigg(\bigg(\frac{D}{p}\bigg)+\alpha_{p}\bigg)\Tr_{D,d}(j_{N,1})\pmod{p}.
\end{equation*}
Finally, we prove the last assertion. Let $A_{m}^{(N)}(D,d)$ denote the coefficient of $q^{D}$ of $f_{N,d}|T(m)$. Then by \cite[Corollary 4.2]{JKK}, we have
\begin{equation*}
A_{p^{r}}^{(N)}(D,d)=\sum\limits_{t=0}^{r}\bigg(\frac{-d}{p}\bigg)^{r-t}A^{(N)}(D,p^{2t}d).
\end{equation*}
On the other hand, by \eqref{Trcongruence} and \cite[Theorem 1.15]{KS}, we have $\frac{1}{\sqrt{D}}\Tr_{D,d}(j_{N,p^{r}})=c_{p^{r}}(D,1)=A_{p^{r}}^{(N)}(D,d)$. Since $A^{(N)}(D,p^{2t}d)=\frac{1}{\sqrt{D}}\Tr_{D,p^{2t}d}(j_{N,1})$, we obtain the desired result.
\end{proof}

\section{Proof of Theorem \ref{counter}}

\begin{proof}
Let $N$ be non-squarefree and $p$ be an odd prime such that $p^{2}|N$. By the Drinfeld-Manin theorem \cite{Drin,Man}, the finiteness of the rational cuspidal subgroup of the Jacobian $J_{0}(N)$ of the modular curve $X_{0}(N)$ implies that there exists a modular function $f$ whose divisor is $n\big((e/p)-(\i)\big)$ where $e/p$ is the cusp of $\G_{0}(N)$ and $n\in\Z_{>0}$. Note that by \cite[Corollary 8]{RW}, if such a function satisfies $f\in\Z((q))$, then it should be an eta-quotient. Hence, we have a contradiction because of the criterion for determining the vanishing orders of an eta-quotient at the cusps(see, \cite{Ligozat} or \cite[Proposition 5.9.3]{CS}). Therefore, we conclude that such a function $f$ has nonintegral Fourier coefficiets. By replacing $e$ with $e+pm$ for some integer $m$, we may assume that $e$ is coprime to $N$ (this is possible because $(e,p)=1$). By replacing $f$ with $f\Delta^{n}$, where $\Delta(\t)=\eta^{24}(\t)$, we may assume that $f$ has zeros only at the cusp $e/p$ and has no poles elsewhere. Note that since $\Delta$ is a multiplicative Hecke eigenform (see \cite[Theorem 1.13]{KS}), $f\Delta^{n}$ is still a multiplicative Hecke eigenform if $f$ is a multiplicative Hecke eigenform. Let $e_{1}/p$ be a cusp of $\G_{0}(N)$ that is $\G_{0}(N)$-inequivalent to the cusp $e/p$ (it exists because $p^{2}|N$). Similar to the case of taking $e/p$, we may assume that $e_{1}$ is coprime to $N$. Take $\ell=e_{1}\overline{e}\in\Z$ where $\overline{e}$ is the inverse of $e$ modulo $N$. From the definition of the multiplicative Hecke operator, we see that $\ell e/p$ is a zero of $f|\T(\ell)$. Then, cusp $\ell e/p$ is $\G_{0}(N)$-equivalent to cusp $e_{1}/p$ if and only if
\begin{equation*}
y\ell e\equiv e_{1}+jp\pmod{N} \;\; \text{ and } \;\; p\equiv yp\pmod{N}
\end{equation*}
has integer solutions $y$ and $j$ with $(y,N)=1$ (see \cite[Proposition 3.8.3]{DS}). Clearly, we take $y=1$ and the first congruence equation is simplified as
\begin{equation*}
jp\equiv \ell e-e_{1}=0\pmod{N}.
\end{equation*}
Therefore, we have a solution $j=0$, which implies that $f|\T(\ell)$ has zero at the cusp $e_{1}/p$. Thus, $f$ is not a multiplicative Hecke eigenform.
\end{proof}

\bibliographystyle{abbrv}
\bibliography{twisted}

\end{document}